\documentclass[11pt, a4paper]{amsart}
\usepackage{amsmath}
\usepackage{amsfonts}
\usepackage{amssymb}
\usepackage{amsthm}
\usepackage{mathrsfs}
\usepackage{enumerate}
\usepackage{hyperref}
\usepackage{cite}

\usepackage{microtype}

\usepackage{afterpage}

\usepackage[T1]{fontenc}

\theoremstyle{plain}
\newtheorem{Th}{Theorem}[section]
\newtheorem{Lemma}[Th]{Lemma}

\newtheorem{Prop}[Th]{Proposition}

\theoremstyle{definition}
\newtheorem{Def}[Th]{Definition}

\newtheorem{Rem}[Th]{Remark}
\newtheorem{?}[Th]{Problem}

\newcommand{\diam}{{\rm{diam}}}

\newcommand{\dist}{\operatorname{dist}}

\newtheorem*{acknowledgement}{Acknowledgement}

\title[Equivalence of Collet--Eckmann conditions]{Equivalence of Collet--Eckmann conditions for slowly recurrent rational maps}
\author{Mats Bylund}

\address{Centre for Mathematical Sciences, Lund University, Box 118, 221 00 Lund, Sweden}
\email{mats.bylund@math.lth.se}


\begin{document}

\begin{abstract}
In this short note we observe that within the family of slowly recurrent rational maps on the Riemann sphere, the Collet--Eckmann, second Collet--Eckmann, and topological Collet--Eckmann conditions are equivalent and also invariant under topological conjugacy.
\end{abstract}

\maketitle

\section{Introduction}
The Collet--Eckmann condition first appeared in the seminal papers by P.~Collet and J.-P.~Eckmann \cite{CE1,CE2} where they studied chaotic behaviour of certain non-uniformly expanding maps on the interval. This condition, which requires exponential growth of the derivative along the critical orbit(s), was later introduced in \cite{P1} to the study of holomorphic (rational) maps on the Riemann sphere. The Collet--Eckmann condition, which often implies the existence of absolutely continuous invariant measures with strong ergodic properties, is known to be abundant in both the real \cite{BC1,BC2,AM} and complex \cite{Rees, Asp} settings. A related and purely topological condition was introduced in \cite{PR1}, where it was proved to be implied by the Collet--Eckmann condition. Much work has been made to identify the relationships between the \emph{Collet--Eckmann condition} (abbr.~CE), the \emph{second Collet--Eckmann condition} (abbr.\ CE2), and the \emph{topological Collet--Eckmann condition} (abbr.\ TCE) (see below for definitions). Notably these conditions are known to be equivalent within the family of unicritical maps (see \cite{PR-LS} and references therein). In \cite{PR-LS} examples are given of maps which satisfy TCE but not CE and/or CE2, and maps which satisfy CE but not CE2, and vice versa. The main problem that arises is when critical points come close to other critical points of high multiplicity. By assuming a recurrence condition of the critical orbits, known as the \emph{slow recurrence condition} (abbr.\ SR), we observe in this note that these conditions become equivalent; in a sense slow recurrence takes the r\^ole of unicritical. The slow recurrence condition is defined as follows.

\begin{Def}
A rational map $f \colon \hat{\mathbb{C}} \to \hat{\mathbb{C}}$ of degree $\geq 2$ is said to satisfy the \emph{slow recurrence condition} if for each $\alpha > 0$ there exists $C > 0$ such that, for every critical point $c \in \operatorname{Crit}(f) \cap J(f)$,
\[
\operatorname{dist}\left(f^n (c), \operatorname{Crit}(f)  \cap J(f)\right) \geq Ce^{-\alpha n} \quad (n \geq 1).
\]
\end{Def}
\begin{Rem}
Note that if $f$ satisfies SR then no critical point is mapped onto another critical point.
\end{Rem}

The SR condition is generally believed to be a typical property among rational Collet--Eckmann maps, and it should be noted that within the real quadratic family this is known to be true due to a result by A.~Avila and C.~G.~Moreira \cite{AM}. In fact they proved that for a typical non-hyperbolic (the critical point does not tend to an attractive cycle) real quadratic map $F$ one has
\[
\operatorname{dist}(F^n (c),c) \geq \frac{C}{n^{1+\epsilon}} \quad (n \geq 1),
\]
for any $\epsilon > 0$ and $C = C(\epsilon) > 0$ a constant. Moreover, in the multimodal setting, B.~Gao and W.~Shen \cite{GS14} proved that for one-parameter families the slow recurrence condition is satisfied on a set of positive Lebesgue measure. We also mention that for complex unicritical polynomials $z \mapsto z^d + c$, it follows from a result by J.~Graczyk and G.~\'Swi\c atek \cite{GS15} that the slow recurrence condition is satisfied for a typical parameter $c$ with respect to harmonic measure on the boundary of the connectedness locus.

The SR condition is also natural in the sense that CE+SR is invariant under topological conjugacy, as was observed by H.~Li (Theorem~A.1 in \cite{Li}, see also \cite{LW06}). The short proof of the invariance is given at the end of this note.

The following is our main observation.
\begin{Prop}\label{mainProp}
Within the family of slowly recurrent rational maps of degree $\geq 2$ on the Riemann sphere, CE, CE2, and TCE are equivalent. Moreover, these conditions are invariant under topological conjugacy.
\end{Prop}
In \cite{PR-LS} examples of real polynomials of degree 5 that satisfy CE but not CE2 (and vice versa) are given, and also examples of real polynomials of degree 3 that satisfy TCE but neither CE nor CE2. We therefore conclude that none of these examples satisfy SR. 

We make a final remark that it would be interesting to investigate the set of rational maps satisfying TCE+SR. Indeed if almost every topological Collet--Eckmann map is slowly recurrent then TCE and CE are equivalent up to a set of measure zero.

Below we indicate the changes in three already existing lemmas in order to reach the above stated result of Proposition~\ref{mainProp}. For completeness we provide the minimal of definitions and proofs, but refer to the relevant articles for greater detail. Throughout this note the standing assumption is that $f$ is a slowly recurrent rational map on the Riemann sphere $\hat{\mathbb{C}}$ of degree $\geq 2$, and with $f^n$ we mean $f$ composed with itself $n$ times. We let $B(z,r) = \{w : \operatorname{dist}(z,w) < r\} \subset \hat{\mathbb{C}}$ denote the disk of radius $r > 0$ centred at $z$, and we let $\operatorname{Crit}'(f) = \operatorname{Crit}(f) \cap J(f)$ with $\operatorname{Crit}(f)$ the set of critical points of $f$, and $J(f)$ the Julia set of $f$. Distances, diameters, and derivatives are taken with respect to the spherical metric on $\hat{\mathbb{C}}$.

\section{Equivalence of CE+SR and CE2+SR}
The Collet--Eckmann condition and second Collet--Eckmann condition are defined as follows.

\begin{Def}
A rational map $f \colon \hat{\mathbb{C}} \to \hat{\mathbb{C}}$ of degree $\geq 2$ without parabolic periodic points is said to satisfy the \emph{Collet--Eckmann condition} (CE) if there exist constants $\lambda_1 > 1$ and $C_1 > 0$ such that, for each critical point $c \in \operatorname{Crit}'(f)$,
\[
\vert (f^n)^\prime (f(c)) \vert \geq C_1 \lambda_1^n \quad (n \geq 0).
\]
\end{Def}

\begin{Def}
A rational  map $f \colon \hat{\mathbb{C}} \to \hat{\mathbb{C}}$ of degree $\geq 2$ is said to satisfy the \emph{second Collet--Eckmann condition} (CE2) if there exist constants $\lambda_2 > 1$ and $C_2 > 0$ such that, for every $n \geq 1$ and every $w \in f^{-n}(c)$ for $c \in \operatorname{Crit}'(f)$ not in the forward orbit of other critical points, 
\[
\vert (f^n)^\prime(w) \vert \geq C_2 \lambda_2^n.
\]
\end{Def}
In \cite{GS} it was proved that these two conditions are equivalent for critical points of maximal (dynamical) multiplicity. This was achieved through the so-called \emph{(reversed) telescope construction}. At the heart of these techniques lies the \emph{shrinking neighbourhoods} (first introduced in \cite{P1}) which are defined as follows. Fix a decreasing sequence of positive real numbers $(\delta_n)$ satisfying $\prod_{n}(1-\delta_n) > 1/2$. Let $B_r = B(z,r)$, and consider a sequence $(f^{-n}(z))$ of consecutive preimages of $z$. With $\Delta_n = \prod_{k < n}(1-\delta_k)$, the $n^\text{th}$ shrinking neighbourhoods of $z$ are now defined as $U_n = \operatorname{Comp}_{f^{-n}(z)} f^{-n}B_{r\Delta_n}$ and $U_n' = \operatorname{Comp}_{f^{-n}(z)} f^{-n}B_{r\Delta_{n+1}}$. Here, $\operatorname{Comp}_w$ denotes the connected component containing $w$. With the right \emph{scale} around each critical point, using these shrinking neighbourhoods, one gets distortion and expansion estimates. The scale is defined by the choice of two positive numbers $R' \ll R \ll 1$, and the correct choice of these two numbers is crucial for the local analysis. We refer to \cite{GS} for details; for our purposes it is enough to keep in mind that these two numbers are fixed throughout the analysis.
\subsection{CE+SR\texorpdfstring{$\implies$}{text}CE2+SR}
Let $(c_{-k})_{k=1}^n$ be a sequence of consecutive preimages of $c_0 = c \in \operatorname{Crit}'(f)$ of length $n \geq 1$, i.e. $f(c_{-k}) = c_{-k+1}$ and $f^k(c_{-k}) = c$. In \cite{GS}, the authors inductively define an increasing sequence of numbers $0 = n_0 < n_1 < \cdots < n_m = n$, and each (backward) orbit of length $n_{k+1}-n_k$ is classified as either a \emph{type $1$}, \emph{type $2$}, or \emph{type $3$} orbit. For orbits of type $\dots 2 , \dots 3,$ or $1\dots 13$ (one reads from the right), one has exponential growth of the derivative (a $12$ block is not allowed by construction). The only problem thus arise when a given backward orbit begins with a block of $1$'s which is not preceded by a $3$. For clarity we give the definition of a type 1 orbit.
\begin{Def}
A sequence $z_0 = z, z_{-1} \in f^{-1}(z),\dots, z_{-n} \in f^{-n}(z)$ of consecutive preimages of $z$ is of the first type with respect to the critical points $c'$ and $c''$ if
\begin{enumerate}[1)]
\item Shrinking neighbourhoods $U_k$ for $B(z,r)$, $1 \leq k \leq n$, avoid critical points for some $r < 2R'$,
\item The critical point $c'' \in \partial U_n$,
\item The critical value of $c'$ is close to $z$ with $f(c') \in B(f(z),R)$.
\end{enumerate}
\end{Def}
The situation of having a block of $1$'s not preceded by a $3$ can only happen in the beginning, and given such a situation the authors prove that there is a constant $\lambda > 1$ such that
\begin{equation}\label{CECE2}
\vert (f^n)^\prime (c_{-n}) \vert^{\mu_{\max}} \geq \operatorname{const} \lambda^n r_1^{\mu_{\max}-\mu(c)},
\end{equation}
where $\mu(c)$ is the multiplicity of $c$, $\mu_{\max} = \max_{c\in \operatorname{Crit}'(f)} \mu(c)$, and $r_1 < 2R'$ is the radius of a disk centred at $c$. Here $r_1$ can not be chosen freely in order for the inductive definition of the $n_k$'s to work, thus for large $n$ (\ref{CECE2}) might not yield expansion. The authors assume $\mu(c) = \mu_{\max}$ and in doing so prove that CE implies CE2 for critical points of maximal multiplicity (Proposition~1 in \cite{GS}). If one assumes SR this problem is easily seen to vanish since the slow recurrence condition dictates how small $r_1$ can be.
\begin{Lemma}
If a slowly recurrent rational map $f \colon \hat{\mathbb{C}} \to \hat{\mathbb{C}}$ of degree $\geq 2$ satisfies CE then it satisfies CE2.
\end{Lemma}
\begin{proof}
Suppose the situation is as described above, and let $n_1$ be the length of the first type $1$ orbit. Per definition of a type $1$ orbit there exists a critical point $c'' \in \partial U_{n_1}$ which is mapped into $B(c,r_1)$. From SR we get that
\[
r_1 \geq \operatorname{dist}(f^{n_1}(c''),c) \geq C e^{-\alpha n_1}.
\]
Since $n_1 \leq n$, inserting the above in (\ref{CECE2}) we find that
\[
\vert (f^n)^\prime (c_{-n}) \vert^{\mu_{\max}} \geq \operatorname{const} \lambda^n \left(Ce^{-\alpha n}\right)^{\mu_{\max} - \mu(c)} \geq C_2 \lambda_2^n,
\]
where we can make $\lambda_2$ arbitrarily close to $\lambda$ by decreasing $\alpha$ (and thus also decreasing $C_2$).
\end{proof}

\subsection{CE2+SR\texorpdfstring{$\implies$}{text}CE+SR}
Pick $c \in \operatorname{Crit}'(f)$, fix $n$, and consider a sequence of images $z_0 = f^n(f(c)), z_{-1} = f^{n-1}(f(c)),\dots, z_{-(n+1)} = c$. Similarly as in the previous case, the authors inductively define an increasing subsequence $n_0 < n_1 < \dots < n_m = n$. Here $n_0$ is the smallest positive integer such that $z_{-(n_0 + 1)}$ is in the $R$-neighbourhood of some critical point. Due to the \emph{exponential shrinking of components} (see below for a definition), which is implied by CE2 (see \cite{PR-LS}), one can prove that during this last orbit of length $n_0$ one has expansion. (In \cite{GS} \emph{R-expansion} was taken as an assumption.) The conditions imposed on $n_j$, $j \neq 0$, are as follows:
\begin{enumerate}[I)]
\item The sequence $z_{-n_{j-1}},\dots,z_{-n_j}$ is of the \emph{first reversed type},
\item Some critical point $c^{(j)} \in B(z_{-(n_j+1)},R)$.
\end{enumerate}
The definition of a \emph{first reversed type} orbit is as follows.
\begin{Def}
A sequence $z_0 = z, z_{-1} \in f^{-1}(z),\dots, z_{-n} \in f^{-n}(z)$ of consecutive preimages of $z$ is of the reversed first type with respect to two critical points $c'$ and $c''$ if
\begin{enumerate}[1)]
\item Shrinking neighbourhoods $U_k$ for $B(z_{-1},r)$, $1 \leq k \leq n-1$, avoid critical points,
\item $\operatorname{dist}(z_{-1},c') = r/2 < R$,
\item $c'' \in U_n$.
\end{enumerate}
\end{Def}
The authors continue and prove (Proposition~5 in \cite{GS}) that there is a constant $\lambda > 1$ such that 
\begin{equation}\label{CE2CE}
\vert (f^n)^\prime(f(c))\vert \geq \operatorname{const} \lambda^n \left(\operatorname{diam}(U_m)\right)^{\mu_{\max} - \mu(c)}.
\end{equation}
Here $\operatorname{diam}(U_m)$ is the diameter of a shrinking neighbourhood around $c$. As in the previous case, this factor might interfere with expansion for large $n$ unless $c$ is assumed to be a critical point of maximal multiplicity $\mu(c) = \mu_{\max}$. Again, assuming SR, we get a lower bound for the diameter.
\begin{Lemma}
If a slowly recurrent rational map $f \colon \hat{\mathbb{C}} \to \hat{\mathbb{C}}$ of degree $\geq 2$ satisfies CE2 then it satisfies CE.
\end{Lemma}
\begin{proof}
It is given that $f^m(c) \in B(c',R)$ for a critical point $c'$, and $U_m$ is the shrinking neighbourhood of $B(f^m(c),r)$ of radius $r = 2\operatorname{dist}(f^m(c),c')$. By definition of a reversed type 1 orbit $f^{-(m-1)} \colon B(f^m(c),r/2) \to f(U_m)$ is univalent, and with an application of Koebe's~$\frac{1}{4}$-lemma we find that
\[
\operatorname{diam}(U_m) \geq \operatorname{diam}(f(U_m)) \geq \frac{1}{C}r\vert (f^{m-1})^\prime (f(c))\vert^{-1}.
\]
(Here $C > 1$ is a constant depending on the scale $R$ we are working with, and it shows up since we are adapting Koebe's~$\frac{1}{4}$-lemma to the spherical metric.) The first inequality follows since $c \in U_m$ and thus the image of $U_m$ under $f$ is contracted. Since $r/2 = \operatorname{dist}(f^m(c),c')$, invoking SR and that $m \leq n$, we find by inserting the above in (\ref{CE2CE}) that
\[
\vert (f^n)^\prime(f(c))\vert \geq \operatorname{const}\left[\frac{1}{C} \vert (f^{m-1})^\prime (f(c))\vert^{-1}\right]^{\mu_{\max} - \mu(c)} \lambda^n \left(Ce^{-\alpha n}\right)^{\mu_{\max} - \mu(c)}.
\]
We observe that $\vert (f^{m-1})^\prime(f(c)) \vert \leq K$ with $K = K(R)$ an absolute constant depending on the choice of $R$. Indeed, for each critical point $c$ under consideration, and for a fixed $R$, there exists a unique smallest integer $m = m(c,R)$ for which $f^m(c) \in B(c'',R)$, for some critical point $c''$. We simply let 
\[
K = \max_{c \in \operatorname{Crit}'(f)} \vert (f^{m(c,R)-1})^\prime (f(c)) \vert.
\]
Thus we get that
\[
\vert (f^n)^\prime(f(c))\vert \geq C_1 \lambda_1^n,
\]
where we can make $\lambda_1$ arbitrarily close to $\lambda$ by decreasing $\alpha$ (and thus also decreasing $C_1$).
\end{proof}

\section{Equivalence of CE+SR and TCE+SR}
The \emph{topological Collet--Eckmann condition} for rational maps on the Riemann sphere was first introduced in \cite{PR1} and is defined as follows. Recall that for a connected set $\Omega$, $\operatorname{Comp}_w g^{-1}(\Omega)$ denotes the connected component of $g^{-1}(\Omega)$ containing $w$.
\begin{Def}
A rational map $f \colon \hat{\mathbb{C}} \to \hat{\mathbb{C}}$ of degree $\geq 2$ is said to satisfy the \emph{topological Collet--Eckmann condition} (TCE) if there exist $M \geq 0$, $P \geq 0$ and $r > 0$ such that for every $z \in J(f)$ there exists a strictly increasing sequence of positive integers $n_j$, for $j = 1,2, \dots$, such that $n_j \leq P \cdot j$, and for each $j$
\[
\# \{i : 0 \leq k < n_j, \operatorname{Comp}_{f^k(z)} f^{-(n_j-k)} \left(B(f^{n_j}(z),r)\right) \cap \operatorname{Crit} \neq \emptyset\} \leq M.
\]
\end{Def}
Since TCE is formulated purely in topological terms it is a topological invariant. One of the useful properties of this condition is its many equivalent formulations (see \cite{PR-LS} and also \cite{PR-L,R-L10}). Here we make use of the following equivalent condition.
\begin{Def}
A rational map $f \colon \hat{\mathbb{C}} \to \hat{\mathbb{C}}$ of degree $\geq 2$ is said to satisfy \emph{exponential shrinking of components} (ExpShrink) if there exists $\lambda_{\operatorname{Exp}} > 1$ and $r_{\operatorname{Exp}} > 0$ such that for every $x \in J(f)$, every $n > 0$, and every connected component $W$ of $f^{-n}(B(x,r_{\operatorname{Exp}}))$
\[
\diam(W) \leq (\lambda_{\operatorname{Exp}}^{-1})^n.
\]
\end{Def}
It was first proved in \cite{PR1} that CE implies TCE, and in \cite{PR2} it was proved that under the assumption that for every $c \in \operatorname{Crit}'(f)$ whose forward trajectory does not meet any other critical point 
\begin{equation}\label{nonrecurrent}
\operatorname{cl}\bigcup_{n > 0} f^n(c) \cap \left(\operatorname{Crit}(f)\smallsetminus \{c\}\right) = \emptyset,
\end{equation}
TCE implies CE. This latter result clearly implies that CE+(\ref{nonrecurrent}) is a topological invariant; in particular CE is a topological invariant in the case of unicritical maps. Another proof of this result was obtained in \cite{P2}. We will effectively replace condition (\ref{nonrecurrent}) with SR, thus proving that TCE+SR implies CE+SR. This constitutes an obvious modification in the proof of Lemma~4.5 in \cite{P2}; for completeness we give a sketch of this proof. (See also Proposition~3.4 in \cite{Li}.)
\begin{Lemma}
If a slowly recurrent rational map $f \colon \hat{\mathbb{C}} \to \hat{\mathbb{C}}$ of degree $\geq 2$ satisfies ExpShrink then it satisfies CE.
\end{Lemma}

\begin{proof}
Let $\alpha$ be the exponent in SR and let $n_0 = n_0(\alpha)$ be large enough such that for every $n \geq n_0$
\[
\operatorname{dist}(f^j(f(c)), \operatorname{Crit}(f)) > e^{-2\alpha n} \quad (j = 0,1,\dots, n-1).
\]
This condition is assumed in Lemma~4.5 \cite{P2}, and the proof now continues as follows. Fix $\epsilon > 0$ arbitrary and let 
\[
s = \left[ \frac{-\log \epsilon}{\log \lambda_{\operatorname{Exp}}} + \frac{2\alpha n}{\log \lambda_{\operatorname{Exp}}} \right] + 1,
\]
where $\left[x\right]$ denotes the integral part of $x$. By ExpShrink we have that for all $0 \leq j \leq n$
\begin{align*}
\operatorname{diam}\left( \operatorname{Comp}_{f^{n-j}(f(c))} f^{-s-j}\left(B(f^{n+s}(f(c)),r_{\operatorname{Exp}})\right) \right) &\leq (\lambda_{\operatorname{Exp}}^{-1})^{s+j} \\
&\leq (\lambda_{\operatorname{Exp}}^{-1})^{s} \\
&\leq \epsilon e^{-2\alpha n}.
\end{align*}
Let $B = B(f^n(f(c)),r_{\operatorname{Exp}} e^{-3\alpha M n})$, where
\[
M = \left[\frac{\log \sup_{\hat{\mathbb{C}}}\vert f^\prime \vert}{\log \lambda_{\operatorname{Exp}}} \right] + 1.
\]
Then for $n$ large enough we get that
\[
B \subset \operatorname{Comp}_{f^n(f(c))} f^{-s}\left(B(f^{n+s}(f(c)), r_{\operatorname{Exp}}) \right).
\]
Let $W_n = \operatorname{Comp}_{f(c)} f^{-n}(B)$. Then there exists $w \in W_n$ such that
\[
\vert (f^n)^\prime(w) \vert \geq \frac{\operatorname{diam} B}{\operatorname{diam} W_n} \geq (2r_{\operatorname{Exp}} e^{-3\alpha M n})\lambda_{\operatorname{Exp}}^n.
\]
If $\epsilon$ is sufficiently small we have distortion and can switch from $w$ to $f(c)$, hence
\[
\vert (f^n)^\prime(f(c)) \vert \geq \lambda_1^n,
\]
where we can make $\lambda_1$ arbitrarily close to $\lambda_{\operatorname{Exp}}$ by decreasing $\alpha$ and $\epsilon$ (and thus increasing $n_0$).
\end{proof}

\section{Topological invariance}
We finish by giving the short proof of the topological invariance, as outlined in Theorem~A.1 \cite{Li}.
\begin{Lemma}
Let $f$ and $g$ be topologically conjugated rational maps on the Riemann sphere of degree $\geq 2$. If $f$ satisfies TCE+SR then so does $g$.
\end{Lemma}
\begin{proof}
Since $f$ is TCE+SR it is CE+SR and therefore, by Theorem~A in \cite{PR2}, the conjugacy is quasi-conformal and therefore bi-H\"older. Let $h$ denote this conjugacy, and let $A > 0$ and $B > 0$ be the associated constant and exponent from the H\"older condition, respectively. Let $c_1'$ and $c_2'$ be distinct critical points of $g$; then $c_1 = h^{-1}(c_1')$ and $c_2 = h^{-1}(c_2')$ are distinct critical points of $f$. Since $h$ preserves TCE, $g$ is at least TCE. The fact that $g$ is also SR follows from
\begin{align*}
A \dist\left(g^n(c_1'),c_2'\right)^B &\geq \dist\left(h^{-1}(g^n(c_1')),h^{-1}(c_2')\right) \\
&= \dist\left(f^n(c_1),c_2\right) \\
&\geq C e^{-\alpha n}.
\end{align*}
\end{proof}

\begin{small}
\begin{acknowledgement}
I thank M.~Aspenberg and W.~Cui for discussions.
\end{acknowledgement}
\end{small}
\bibliographystyle{alpha}
\bibliography{references}

\end{document}